\documentclass[a4paper]{amsart}
\usepackage{color}
\usepackage{amssymb}
\usepackage{enumitem}
\usepackage{comment}
\usepackage{graphicx}
\usepackage{pstricks}
\usepackage{tikz}
\usetikzlibrary{patterns}
\usepackage{accents}
\usepackage{xspace}
\usepackage{algorithm}
\usepackage{algpseudocode}

\numberwithin{equation}{section}

\setlist[itemize]{leftmargin=4ex}

\theoremstyle{plain}
\newtheorem{thm}{Theorem}[section]
\newtheorem{prop}[thm]{Proposition}

\theoremstyle{definition}
\newtheorem{definition}[thm]{Definition}

\newcommand{\R}{\mathbb{R}}
\newcommand{\Poly}{\mathbb{P}}
\newcommand{\diam}{\mathrm{diam}}

\newcommand{\tri}{\mathcal{T}}              
\newcommand{\faces}{\mathcal{F}_\Omega} 
\newcommand{\Afaces}{\mathcal{F}}           
\newcommand{\el}{K}                         
\newcommand{\fa}{F}                         
\newcommand{\pair}[1]{\omega_{#1}}      
\newcommand{\path}{\gamma}
\newcommand{\ShapePar}{\sigma}                          
\newcommand{\FEspace}{S}                        
\newcommand{\LFEspace}[1]{\FEspace|_{#1}}       
\newcommand{\nodes}{\mathcal{N}}
\newcommand{\verts}{\mathcal{V}}

\newcommand{\lNnn}{\lVert {\hskip -0.1em} \lvert}
\newcommand{\rNnn}{\rvert {\hskip -0.1em} \rVert}
\newcommand{\tnorm}[1]{\lNnn #1 \rNnn}      
\newcommand{\norm}[1]{\left\|#1\right\|}    

\newcommand{\veps}{\varepsilon}

\newcommand{\interp}{\Pi}                   

\DeclareMathOperator{\Div}{div}

\newcommand{\vphi}{\varphi}

\begin{document}
\title[Quasi-monotonicity and Robust Localization]{Quasi-monotonicity and Robust Localization with Continuous Piecewise Polynomials}
\author[F.~Tantardini]{Francesca Tantardini}
%
\author[R.~Verf\"urth]{R\"udiger Verf\"urth}
\address[R\"udiger Verf\"urth]
 {Fakult\"at f\"ur Mathematik\\
  Ruhr-Universit\"at Bochum\\
  Uni\-ver\-si\-t\"ats\-stra{\ss}e 150\\
  44801 Bochum\\
  Germany}
\email[R\"udiger Verf\"urth]{ruediger.verfuerth@ruhr-uni-bochum.de}
\urladdr[R\"udiger Verf\"urth]{homepage.ruhr-uni-bochum.de/ruediger.verfuerth}

\begin{abstract}
We consider the energy norm arising from elliptic problems with discontinuous piecewise constant diffusion. We prove that under the quasi-monotonicity property on the diffusion coefficient, the best approximation error with continuous piecewise polynomials is equivalent to the $\ell_2$-sum of best errors on elements, in the spirit of \cite{Veeser:16} for the $H^1$-seminorm. If the quasi-monotonicity is violated, counterexamples show that a robust localization does not hold in general, neither on elements, nor on pairs of adjacent elements, nor on stars of elements sharing a common vertex.  
\end{abstract}
\maketitle 

\section{Introduction}
We consider the approximation of a function $u\in H^1(\Omega)$, $\Omega\subset\R^d$,  with continuous piecewise polynomials belonging to a finite element space $S=S^{\ell, 0}(\tri)$ with underlying triangulation $\tri$. 
We measure the error in the seminorm
\begin{equation*}
 \tnorm{v}:=\norm{a^{1/2}\nabla v},
\end{equation*}
where $\norm{\cdot}$ indicates the $L^2$-norm and $a$ is piecewise constant, uniformly positive, and $a_\el:=a|_\el\in\R$ is constant on every element of the triangulation. 
Such a seminorm arises, for example, as energy norm in the elliptic problem with non-smooth diffusion
\begin{equation}
\label{pde}
 -\Div(a\nabla u)=f\quad \textrm{in } \Omega, 
 \qquad u=0 \quad \textrm{on } \partial\Omega.
\end{equation}

Since every function in $S$ is locally a polynomial of degree $\ell$, it is immediate to see that the global best error in $S$ is bounded from below by the $\ell_2$-sum of the local best errors on the elements of the triangulation 
\begin{equation*}
\inf_{v\in S}\tnorm{u-v}^2\geq\sum_{\el\in\tri}a_\el\inf_{p_\el\in \Poly_\ell(\el)}\norm{\nabla(u-p_\el)}^2_\el.
\end{equation*}
For $a=1$ it is proven in \cite{Veeser:16} that the non-trivial converse inequality 
also holds up to a constant dependent on the polynomial degree $\ell$, the dimension $d$ and the shape parameter of the mesh. The interest in such a result  arises in particular because it allows to derive error estimates in terms of piecewise regularity and to identify simple error functionals to drive the tree approximation algorithm of Binev and DeVore \cite{Binev.DeVore:04}.

We would like to understand if and how the jumps in the coefficient $a$ influence 
the validity of 
\begin{equation}
\label{loc-intro}
\inf_{v\in S}\tnorm{u-v}^2\lesssim \sum_{\el\in\tri}a_\el\inf_{p_\el\in\Poly_\ell(\el)}\norm{\nabla(u-p_\el)}^2_\el, \quad\forall u\in H^1(\Omega).
\end{equation}
We show with an example that there exists a configuration of the coefficient $a$ on the domain $\Omega$ such that \eqref{loc-intro}
 does not hold robustly in the ratio $\alpha:=\min a/\max a$ between minimal and maximal value of $a$. This means that, contrary to the case when $a=1$, the continuity constraint prevents $S$ from fully exploiting the local approximation potential of $\Poly_\ell(\el)$.

A milder continuity constraint  in $d=2$ is provided by the non-conforming Crouzeix-Raviart space, consisting of those piecewise affine functions that are continuous only at the midpoint of the edges of $\tri$. Replacing $S$ by the Crouzeix-Raviart space and the $\tnorm{\cdot}$-norm  by its broken counterpart,  \eqref{loc-intro} holds independently of $a$ and even with an equality, see e.g.\ \cite{Zanotti:13}.  

When assuming a condition on the coefficient $a$ called quasi-monotonicity, the localization \eqref{loc-intro} holds robustly in $\alpha$ for the conforming space $S$ too. Quasi-monotonicity means that for every couple of elements that share at least a vertex, there exists a connecting path of adjacent elements such that $a$ is monotone along the path (cf. Definition \ref{D:quasi-mon} below). A typical example of violation of quasi-monotonicity is the checkerboard pattern, see also Figure \ref{F:def_qm} below. This property comes into play also in the a posteriori error analysis of the finite element Galerkin solution to \eqref{pde}, see \cite{Bernardi.Verfuerth:00}, in the validity of weighted Poincar\'e inequalities  \cite{Pechstein.Scheichl:13} and in domain decomposition methods \cite{Dryja.Sarkis.Widlund:96}. 

The question arises, if in the absence of quasi-monotonicity, it is possible to localize on bigger subdomains than elements. The above-mentioned example shows that it is not possible to robustly localize the best error on pairs of elements. To see this, note that the diffusion is always quasi-monotone on pairs of elements. Thus a robust localization of the best error on pairs of elements would entail the robust localization on elements. More strikingly, the above-mentioned example can be modified to show that even on stars of elements sharing a common vertex the localization of the best error is not robust with respect to the parameter $\alpha$. Thus, the situation is completely different to the localization of the best error in the $H^{-1}$-norm. For this norm, it is proved in \cite{TVV:-1} that a localization on elements or pairs of elements is not possible while it can be achieved on stars of elements.

This note is organized as follows. In \S \ref{S:aux_res} we set the notation, while in \S \ref{S:loc-quasi-mon} we prove the robust localization on elements under quasi-monotonicity and consider also a reaction-diffusion norm. In \S \ref{S:not-rob} we then present two examples that show that a robust localization does not hold in general, neither on elements, nor on pairs of adjacent elements, nor on stars of elements sharing a common vertex.

\section{Notation}
\label{S:aux_res}
%
%
%

\subsection{Triangulation and finite element space}
We denote by $\tri$ a  conforming simplicial mesh of a polyhedral domain 
$\Omega\subset\mathbb{R}^d$, $d\in\mathbb{N}$, by $\Afaces$ the set of 
its $(d-1)$-dimensional faces and by $\verts$ the set of its vertices. 
We assume that $\tri$ is subordinate to $a$, i.e., for every $\el\in\tri$, $a$ is constant on $\el$. 
If $\el\in\tri$ is an element and $\fa\in\Afaces$ is a 
face, we write $|\el|$ and $|\fa|$ for its $d$-dimensional Lebesgue and 
$(d-1)$-dimensional Hausdorff measure, respectively.
For every face $\fa\in \Afaces$, the set  $\pair{\fa} := \bigcup\{\el\in\tri: \partial\el\supseteq\fa \}$
is the union of elements sharing the face $\fa$. We assume that $\tri$ belongs to a family of shape-regular meshes and satisfies
\begin{equation*}
\sup_{\el\in\tri}\frac{h_\el}{\rho_\el}\leq\ShapePar<+\infty,
\end{equation*} 
where $h_\el:=\diam(\el)$ and $\rho_\el$ is the 
maximum diameter of a ball inscribed in $\el$.

The space
\begin{equation*}
 \FEspace
 :=
 \FEspace^{\ell,0}(\tri)
 =
 \{v\in C^{0}(\Omega): \;
  \forall\, \el\in\mathcal{\tri},\, v\in\mathbb{P}_\ell(\el) \}
\end{equation*}
consists of all continuous functions that are piecewise polynomial over $\tri$. 
Given a set $\omega\subset\Omega$, we indicate its restriction with
\begin{equation*}
 \LFEspace{\omega}
 :=
 \{v\in C^0(\omega):
 \exists \tilde{v}\in \FEspace, \ \tilde{v}|_{\omega}=v\}.
\end{equation*}
In particular, for any element $\el\in\tri$, it holds
$\LFEspace{\el}=\mathbb{P}_\ell(\el)$ and, for any $\fa\in\Afaces$, 
we have $\LFEspace{\pair{\fa}} = \{v\in C^{0}(\pair{\fa}): 
\forall \el\in\tri\text{ with }\el\subseteq\pair{\fa}, 
\, v\in\mathbb{P}_\ell(\el) \}$.

The set of nodes of $\FEspace$ is indicated by
\begin{equation*}
 \mathcal{N}
 :=
 \bigcup_{\el\in\tri} \mathcal{N}_\el
\quad\text{with}\quad
 \mathcal{N}_\el
 :=
 \left\{ \sum_{i=0}^d \frac{\alpha_i}\ell a_i : 
  \alpha\in\mathbb{N}_0^{d+1}, \sum_{i=0}^d \alpha_i = \ell
 \right\}
\end{equation*}
whenever $a_0,\dots,a_d$ are the vertices of a simplex $\el$. In accordance 
with the definition of $\mathcal{N}_\el$, a subscript $\el$, 
$\Omega$, etc.\ to $\mathcal{N}$, $\verts$, or $\Afaces$ indicates that only those nodes, vertices
or faces which are contained in the index set are considered.


Moreover, we denote by $\{\phi_z\}_{z\in\mathcal{N}}$ the nodal basis of $\FEspace$, for which each 
$\phi_z$ is uniquely determined by
\begin{equation*}
 \phi_z\in \FEspace
 \qquad\text{and}\qquad
 \forall y\in\mathcal{N} \quad \phi_{z}(y)=\delta_{yz}.
\end{equation*}
We indicate with $\omega_z:=\mathrm{supp} (\phi_z)$ the support of $\phi_z$. 
Given an element $\el\in\tri$, the $L^2(\el)$-dual basis functions 
$\{\psi_z^\el\}_{z\in\mathcal{N}_\el}$ are such that, for every 
$z\in\mathcal{N}_\el$, it holds
\begin{equation*}
 \psi_z^\el \in \mathbb{P}_\ell(\el)
\qquad\text{and}\qquad
 \forall y\in\mathcal{N}_\el \quad
  \int_\el\psi_z^\el\phi_y=\delta_{zy}.
\end{equation*}
We thus have, for every $p\in\mathbb{P}_\ell(\el)$ and for every 
$z\in\mathcal{N}_\el$,
\begin{equation}
\label{psi_prop}
 p(z)=\int_\el p\psi_z^\el.
\end{equation}

\subsection{Scaling properties} 
Here and in what follows we write $A\lesssim B$ and $A\gtrsim B$ with the meaning that there exist constants 
$C$, $c>0$, such that $A\leq C B$ and $A\geq c B$ respectively, where $C$, $c$ may depend on $\ShapePar$, $\ell$ and $d$,
 but are 
independent of other parameters like $h_\el$ and in particular of $\alpha=\min a/\max a$. We write also $A\approx B$ when both
$A\gtrsim B$ and $A\lesssim B$ hold. 

In order to bound the norms of basis and dual basis functions, the standard procedure is to transform to the reference element $\widehat{\el}:=\mathrm{conv hull}\{{\bf{0}}, e_1,\ldots, e_d\}$, see \cite[p.\ 117-120]{Ciarlet:78}.  
For the $L^2$-norm  and the $H^1$-seminorm we obtain
\begin{equation}
\label{scaling-L2}
\norm{\phi_z}_{\el}\approx |\el|^{1/2}, \qquad \norm{\nabla\phi_z}_{\el}\lesssim h_\el^{-1}|\el|^{1/2}, \qquad z\in\nodes_\el.
\end{equation}
Concerning the dual basis functions, we observe that on the reference element, the dual basis of $\{\widehat{\phi}_{\widehat{z}}\}_{\widehat{z}\in\nodes_{\widehat{\el}}}:=\{\phi_z\circ G_\el\}_{z\in\nodes_\el}$ are given by $\widehat{\psi}_{\widehat{z}}=\det(\mathbf{G}_\el)\cdot\psi_z^\el\circ G_\el$, where $\mathbf{G}_\el$ is the non singular matrix associated with one of the affine transformations $G_\el:\R^d\to\R^d$ such that $G_\el(\widehat{\el})=\el$. We therefore have 
\begin{equation}
\label{scale-dual}
\norm{\psi_z^\el}_\el\approx |\el|^{-1/2}.
\end{equation}

\subsection{Quasi-monotonicity}
We recall the definition of the quasi-monotonicity \linebreak[4]property, which plays a fundamental role for the robustness of the localization. 

Given $z\in\nodes$ and $\el$, $\widetilde{\el}\subset\omega_z$ a monotone path 
$\path(\el,\widetilde{\el})$ between $\el$ and $\widetilde{\el}$ is a collection of adjacent elements 
$\el_0,\ldots, \el_m\subset\omega_z$ such that
\begin{itemize}
\item $\el_0:=\el$,  and $\el_m:=\widetilde{\el}$ 
\item $\el_n\cap \el_{n+1}=F_n\in\faces$, $n=0,\ldots, m-1$
\item $a_{\el_n}\leq a_{\el_{n+1}}$, $n=0,\ldots, m-1$
\end{itemize}
that is, the function $a$ is monotone along the path. 
\begin{definition}[Quasi-monotonicity]
\label{D:quasi-mon}
A piecewise constant function $a$ satisfies the quasi-monotonicity property if, for every $z\in\nodes$ and $\el$, 
$\widetilde{\el}\subset\omega_z$, such that $a|_\el\leq a|_{\widetilde{\el}}$ there exists a monotone path $\path(\el,\widetilde{\el})$. 
\end{definition}
Examples where the quasi-monotonicity property is satisfied or not-satisfied are depicted in Figure \ref{F:def_qm}. 
\newrgbcolor{zzttqq}{0.8 0.8 0.8}
\newrgbcolor{ztq}{0.6 0.6 0.6}
\newrgbcolor{zztq}{0.7 0.7 0.7}
 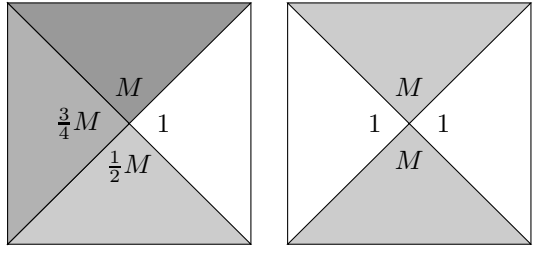
\begin{figure}[htb]
  \begin{center}
\begin{tikzpicture}[scale=0.8](-3.1,-3.1)(3.1,3.1)
 \draw[fill=zzttqq]  (0,0)--(2,-2)--(-2,-2)--(0,0);
\draw [fill=ztq] (0,0)--(2,2)--(-2,2)--(0,0);
 \draw[fill=zztq] (0,0)--(-2,-2)--(-2,2)--(0,0);
  \draw (2,-2)--(2,2);
\node [above] at (0,0.3) {$M$};
\node [below] at (0,-0.3){$\frac{1}2M$};
\node [left] at(-0.3,0){$\frac{3}4M$};
\node[right] at (0.3,0){$1$};
\end{tikzpicture}
\quad
\begin{tikzpicture}[scale=0.8](-3.1,-3.1)(3.1,3.1)
 \draw[fill=zzttqq]  (0,0)--(2,-2)--(-2,-2)--(0,0)--(2,2)--(-2,2)--(0,0);
 \draw (-2,-2)--(-2,2);
  \draw (2,-2)--(2,2);
\node [above] at (0,0.3) {$M$};
\node [below] at (0,-0.3){$M$};
\node [left] at(-0.3,0){$1$};
\node[right] at (0.3,0){$1$};
\end{tikzpicture}
\caption{The quasi-monotonicity property is satisfied (left; $M\geq2$) or not (right; $M\ne1$)}\label{F:def_qm}
 \end{center}
 \end{figure}

\subsection{Trace and Poincar\'e inequalities}
We recall some useful inequalities. For every $v\in H^1(\el)$ with $\int_\el v=0$ the trace inequality
\begin{equation}
\label{trace-ineq}
\norm{v}_\fa\lesssim h_\el\frac{|F|^{1/2}}{|\el|^{1/2}}\norm{\nabla v}, 
\end{equation}
holds for every $\fa\in\Afaces_{\partial\el}$. Moreover, for every $\vphi\in H^1(\omega)$ the Poincar\'e inequality
\begin{equation}
 \label{poinc}
\norm{\vphi-\frac{1}{|\omega|}\int_\omega \vphi}_\omega\lesssim \diam(\omega)\norm{\nabla\vphi}_\omega
\end{equation}
holds. 
The hidden constant in \eqref{poinc} has been explicitly bounded in \cite{Veeser.Verfuerth:12} for a finite element star $\omega=\omega_z$ with $z\in\verts$. 

\section{Robust localizations under quasi-monotonicity}
\label{S:loc-quasi-mon}

\subsection{Pure diffusion norm}\label{S:rob-loc}
In this section, we show that it is possible to robustly localize on elements, if the quasi-monotonicity property \ref{D:quasi-mon} is satisfied.  The proof is an adaptation of the one in \cite{Veeser:16}.

We start by defining a suitable quasi-interpolation operator $\Pi: H^1(\Omega)\to S^{\ell,0}(\tri)$ 
\begin{equation*}
\Pi u=\sum_{z\in \nodes}\Pi u(z)\phi_z,
\end{equation*}  
where the definition of the value of $\Pi u$ in a node is different if $z$ is an interior node or a 
node that couples elements. We indicate with $P_\el\in \Poly_\ell(\el)$ the best approximation to $u$ 
in the $\norm{a^{1/2}\nabla\cdot}$-norm such that 
\begin{equation}
 \label{mean-val}
\int_\el P_\el=\int_\el u
\end{equation}
and set
\begin{equation*}
\Pi u(z):=
  \left\{ 
    \begin{aligned}
    &P_\el(z) &&\text{if }z\in\nodes_{\stackrel{\circ}{\el}}
    \\
    &\int_{F_z} u\psi^{F_z}_z &&\text{if }z\in\nodes\cap\Sigma             
    \end{aligned}
 \right.
\end{equation*}
where $\Sigma:=\cup_{\el\in\tri}\partial\el$ and $F_z$ is a face containing $z$. For every $z\in\nodes$ we choose 
one of the elements $\el_{\max}(z)$ in $\omega_z$ such that
\begin{equation}
\label{Tmax}
 \forall\ \el\subset \omega_z\qquad a|_{\el_{\max}(z)}\geq a|_{\el}.
\end{equation}
In order to get robustness, we take the face $F_z$ as one of the faces of $\el_{\max}(z)$. 
\begin{prop}
\label{P:prop-interp}
Assume the coefficient $a$ satisfies the quasi-monotonicity property \ref{D:quasi-mon}. Then, for every $\el\in\tri$,  there holds
\begin{equation}
\label{prop-interp}
 \tnorm{u-\Pi u}^2_\el\lesssim 
\sum_{\el'\subset \omega_\el} \inf_{P\in \Poly_\ell(\el')}\tnorm{u-P}_{\el'}^2,
\end{equation}
where $\omega_\el:=\cup\{T\in\tri, \el\cap T\neq\emptyset\}$ is the union of elements sharing at least a vertex with $\el$.
\end{prop}
\begin{proof}
We insert $P_\el$ in $\tnorm{u-\Pi u}_\el$ and after a triangle inequality we are left to bound the norm of the difference between $\Pi u$ and the local best approximation $P_\el$. Since
$\Pi u-P_\el\in\Poly_\ell(\el)$ we have 
\begin{equation*}
\tnorm{\Pi u-P_\el}_\el\leq\sum_{z\in\nodes_{\partial \el}}|(\Pi u-P_\el) (z)|\tnorm{\phi_z}_\el,
\end{equation*}
where we further need to bound $(\Pi u-P_\el) (z)$. 
We add and subtract the value of the best approximations with respect to the elements in the monotone path  
$\path(\el,\el_{\max}(z))=\{\el_n\}_{n=0}^m$. We set $\fa_n:=\el_n\cap\el_{n-1}$ and for simplicity we write $P_n$ for $P_{K_n}$ and $K_{\max}$ for $K_{\max}(z)$. 
Thanks to \eqref{mean-val} we can use the Poincar\'e inequality on $K_n$. Together with \eqref{psi_prop},
\eqref{scale-dual} and the trace inequality \eqref{trace-ineq}, we get
\begin{equation*}
\begin{aligned}
&\int_{F_z} u \psi^{F_z}_z - P_\el(z)
=\int_{F_z} u \psi^{F_z}_z - P_{\el_{\max}}(z)
+\sum_{n=1}^m \left( P_n(z)-P_{n-1}(z) \right)
\\
&\qquad= \int_{F_z}(u-P_{\el_{\max}})\psi_z^{F_z}+\sum_{n=1}^m\int_{F_n}(P_n-P_{n-1})\psi_z^{F_n}
\\
&\qquad\leq\norm{u-P_{\el_{\max}}}_{F_z}\norm{\psi_z^{F_z}}_{F_z}
+\sum_{n=1}^m (\norm{u-P_n}_{F_n}
+\norm{u-P_{n-1}}_{F_n})\norm{\psi_z^{F_n}}_{F_n}
\\
&\qquad\lesssim\sum_{n=1}^m \norm{\nabla(u-P_n)}_{\el_n}\frac{h_{\el_n}}{|\el_n|^{1/2}}
\lesssim\sum_{n=1}^m \frac{1}{a_{\el_n}}\tnorm{u-P_n}_{\el_n} \frac{h_{\el_n}}{|\el_n|^{1/2}}.
\end{aligned}
\end{equation*}
Summing over $z$ and recalling \eqref{scaling-L2} as well as the crucial inequality, which holds thanks to the quasi-monotonicity of $a$,
\begin{equation*}
a_\el\leq a_{\el_n}\qquad \forall \el_n\in \path(\el,\el_{\max}(z)), \quad z\in\nodes_{\partial\el},
\end{equation*}
we arrive at 
\begin{align*}
 \tnorm{\Pi u-P_\el}_\el
&\lesssim \sum_{z\in\nodes_{\partial \el}}
\sum_{\el'\in\path(\el,\el_{\max}(z))}\frac{a_\el}{a_{\el'}}\tnorm{u-P_{\el'}}_{\el'} 
\frac{|\el|^{1/2}}{|\el'|^{1/2}}
\frac{h_{\el'}}{h_\el}
\\ &\lesssim  \sum_{\el'\subset\omega_\el}\tnorm{u-P_{\el'}}_{\el'}.\qedhere
\end{align*}
\end{proof}
Since 
\begin{equation*}
\inf_{v\in S^{\ell,-1}(\tri)}\tnorm{u-v}^2\leq \tnorm{u-\Pi u}^2=\sum_{\el\in\tri}\tnorm{u-\Pi u}^2_\el
\end{equation*}
and the number of elements in a patch $\omega_\el$ is bounded in terms of 
$\ShapePar$, from Proposition \ref{P:prop-interp} we derive the following robust localization.
\begin{prop}
\label{P:rob-loc-el}
Assume the coefficient $a$ satisfies the quasi-monotonicity property \ref{D:quasi-mon}. Then, for every $u\in H^1(\Omega)$ there holds
\begin{equation}
 \label{loc-res-grad}
\inf_{v\in S^{\ell,0}(\tri)}\norm{a^{1/2}\nabla(u-v)}^2\lesssim 
\sum_{\el\in\tri}\inf_{P\in \Poly^{\ell}(\el)}a_\el\norm{\nabla(u-P)}^2_\el.
\end{equation}
where the hidden constant is independent of $\alpha=\min a/\max a$.
\end{prop}

\subsection{Reaction-Diffusion norm}

The robust localization on elements in Proposition \ref{P:rob-loc-el} can be combined with the robust localization for the reaction-diffusion norm in \cite{TVV:15} to show a robust localization for the norm
\begin{equation*}
\left(\norm{a^{1/2}\nabla\cdot}^2+\beta\norm{\cdot}^2\right)^{1/2}, \qquad \beta>0.
\end{equation*}
 
We can first separate the error related to the function itself from the error related to the gradient. More precisely, we aim at
\begin{equation}
\begin{aligned}
&\inf_{v\in S^{\ell,0}(\tri)} \left( \norm{a^{1/2}\nabla(u-v)}^2+\beta\norm{u-v}^2 \right) \\
&\qquad\approx 
\inf_{v\in S^{\ell,0}(\tri)}\norm{a^{1/2}\nabla(u-v)}^2+\inf_{w\in S^{\ell,0}(\tri)}\beta\norm{u-w}^2.
\end{aligned}
\label{dec-norms}
\end{equation}
Taking the same function in the infima on the right-hand side, it is trivial to see that the right-hand side is a lower bound for the left-hand one. 
Assuming quasi-monotonicity, the converse inequality is also true up to a constant independent 
of $\alpha$ and $\beta$. To see this, consider the interpolation operator 
$\interp:L^2(\Omega)\to S^{\ell,0}(\tri)$ defined as
\begin{equation}
 \label{def-interp-op}
\interp u:=\sum_{z\in \nodes}\left(\int_{\el_{\max}(z)}u\psi_z^{\el_{\max}(z)}\right)\phi_z,
\end{equation}
where $\el_{\max}$ is defined as in \eqref{Tmax}.
\begin{prop}[Properties of $\interp$]
\label{P:prop-interp-rd}
 The interpolation operator defined in \eqref{def-interp-op} is invariant on $S^{\ell,0}(\tri)$ and $L^2$-stable
 \begin{equation*}
\interp s=s \quad\forall s\in S^{\ell,0}(\tri), \qquad
 \norm{\interp u}\lesssim \norm{u} \quad\forall u \in L^2(\Omega).
\end{equation*}
If, in addition, the quasi-monotonicity property \ref{D:quasi-mon} holds, the operator is energy-stable too
\begin{equation*}
\norm{a^{1/2}\nabla\interp u} \lesssim \norm{a^{1/2}\nabla u} .
\end{equation*}
All hidden constants are independent of $\alpha$ and $\beta$.
\end{prop}

\begin{proof}
Invariance over $S^{\ell,0}(\tri)$ follows readily from \eqref{psi_prop}.

Concerning local stability in $L^2$, for every $\el\in\tri$, by means of \eqref{scaling-L2}--\eqref{scale-dual} we obtain
\begin{equation*}
\begin{aligned}
\norm{\interp u}_\el&\le\sum_{z\in \nodes_\el}\left|\int_{\el_{\max}(z)}u\psi_z^{\el_{\max}(z)}\right|\norm{\phi_z}_\el\\
&\lesssim \sum_{z\in\nodes_\el}\norm{u}_{\el_{\max}(z)}\frac{|\el|^{1/2}}{|\el_{\max}(z)|^{1/2}}
\lesssim\norm{u}_{\omega_\el}.
\end{aligned}
\end{equation*}
Global stability in $L^2$ follows from the finite overlapping of the
patches $\omega_\el$.

Regarding stability in $\norm{a^{1/2}\nabla\cdot}$, for every $\el\in\tri$ we set
\begin{equation*}
 \widehat{\omega}_\el:=\bigcup_{z\in \nodes_\el}\gamma(\el,\el_{\max}(z)).
\end{equation*}
An example of such a set is given in Figure \ref{F:set_poinc} for the polynomial degree $\ell=1$. 
We notice that $ \widehat{\omega}_\el$ is a connected set and define 
\begin{equation*}
c_\el(u):=\frac{1}{|\widehat{\omega}_\el|}\int_{\widehat{\omega}_\el} u.
\end{equation*}
Since $\Pi$ is locally invariant over constants, we can subtract $c_\el(u)$ and exploit 
\eqref{psi_prop}, \eqref{scaling-L2}--\eqref{scale-dual}, the Poincar\'e inequality on $\widehat{\omega}_\el$ as well as $a_\el\leq a_{\el'}$ for every $\el'\subset\widehat{\omega}_\el$, to get
\begin{equation*}
\begin{aligned}
a_\el\norm{\nabla\interp u}_\el&=a_\el\norm{\nabla(\interp u-c_\el(u))}_\el=
a_\el\norm{\nabla\interp (u-c_\el(u))}_\el
\\
&\leq a_\el\sum_{z\in\nodes_\el}\norm{u-c_\el(u)}_{\el_{\max}(z)}\norm{\psi^{\el_{\max}}_z}_{\el_{\max}}\norm{\nabla\phi_z}_\el
\\
&\lesssim  a_\el h_\el^{-1}\norm{u-c_\el(u)}_{\widehat{\omega}_\el}
\lesssim a_\el\norm{\nabla u}_{\widehat{\omega}_\el}\lesssim\norm{a^{1/2}\nabla u}_{\omega_\el}.
\end{aligned}
\end{equation*}
Summing over $\el$ we get global stability by the finite overlapping of the patches $\omega_\el$.
\end{proof}
\newrgbcolor{zzttqq}{0.75 0.75 0.75}
\newrgbcolor{ztq}{0.6 0.6 0.6}
\newrgbcolor{zztq}{0.85 0.85 0.85}
\begin{figure}
 \begin{center}
  \begin{tikzpicture}[scale=0.8](-4.1,-3.1)(4.1,3.1)
\draw[fill=zzttqq] (-2,1)--(-3,0)--(-1,-2)--(0,-1)--(-2,1);
\draw[fill=zztq] (-2,1)--(-1,2)--(1,2)--(2,1)--(1,0)--(-1,0)--(-2,1);
\draw[fill=ztq] (2,1)--(3,0)--(1,-2)--(1,0);
\draw (1,-2)--(0,-1)--(1,0)--(3,0);
\draw (2,-1)--(-1,2)--(-1,-2);
\draw (1,0)--(1,2)--(-2,-1);
\draw (-3,0)--(-1,0);
\node [above] at (0,0.1) {$K$};
\draw[very thick] (3,0)--(2,1)--(1,0)--(3,0);
\fill[pattern=north west lines] (3,0)--(2,1)--(1,0)--(3,0);
\draw [very thick] (-1,0)--(-2,1)--(-3,0)--(-1,0);
\fill[pattern=north west lines] (-1,0)--(-2,1)--(-3,0)--(-1,0);
\draw[very thick] (1,2)--(0,1)--(-1,2)--(1,2);
\fill[pattern=north west lines] (1,2)--(0,1)--(-1,2)--(1,2);  
  \end{tikzpicture}
\quad
  \begin{tikzpicture}[scale=0.8](-4.1,-3.1)(4.1,3.1)
\draw[fill=zzttqq] (-2,1)--(-3,0)--(-1,-2)--(0,-1)--(-2,1);
\draw[fill=zztq] (-2,1)--(-1,2)--(1,2)--(2,1)--(1,0)--(-1,0)--(-2,1);
\draw[fill=ztq] (2,1)--(3,0)--(1,-2)--(1,0);
\draw (1,-2)--(0,-1)--(1,0)--(3,0);
\draw (2,-1)--(-1,2)--(-1,-2);
\draw (1,0)--(1,2)--(-2,-1);
\draw (-3,0)--(-1,0);
\draw[very thick]   (-3,0)--(-1,2)--(1,2)--(3,0)--(-3,0);
\fill[pattern=north west lines] (-3,0)--(-1,2)--(1,2)--(3,0)--(-3,0);
  \end{tikzpicture}
\caption{Example of $\omega_\el$ and $\widehat{\omega}_K$ for $\ell=1$. The different values of $a$ are indicated with different colors following the rule: The darker the region, the bigger the value of $a$. On the left: Element $K$ and elements $K_{\max}(z)$ (hatched), corresponding to the vertices $z$ of $\el$.  
On the right: $\widehat{\omega}_K$ (hatched). }\label{F:set_poinc}
 \end{center}
\end{figure}
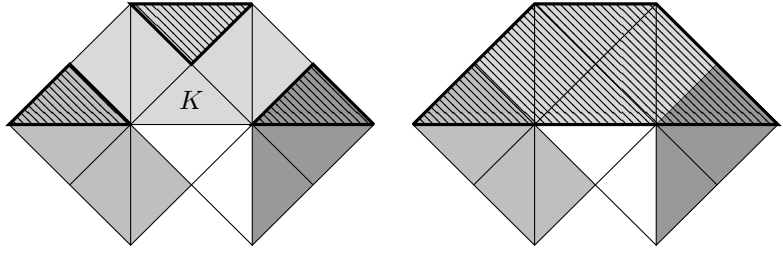

Note that in the proof of the $L^2$-stability neither the choice of $\el_{\max}(z)$ nor the quasi-monotonicity
 property  play a role. Instead, they are crucial in 
the proof of the stability in the $\norm{a^{1/2}\nabla\cdot}$-norm.

 From Proposition \ref{P:prop-interp-rd} it follows immediately that, for every $u\in H^1(\Omega)$,  $\Pi u$ is a near-best approximation in $S^{\ell,0}(\tri)$ in both the $L^2$ and in the $\norm{a^{1/2}\nabla\cdot}$-norm. Hence, we can prove \eqref{dec-norms}: 
\begin{equation*}
\begin{aligned}
&\inf_{v\in S^{\ell,0}(\tri)}\norm{a^{1/2}\nabla(u-v)}^2+\inf_{w\in S^{\ell,0}(\tri)}\beta\norm{u-w}^2
\\
&\qquad\qquad\leq\norm{a^{1/2}\nabla(u-\interp u)}^2+\beta\norm{u-\interp u}^2
\\
&\qquad\qquad\lesssim 
\inf_{v\in S^{\ell,0}(\tri)}\norm{a^{1/2}\nabla(u-v)}^2
  +\inf_{w\in S^{\ell,0}(\tri)}\beta\norm{u-w}^2.
\end{aligned}
\end{equation*}
We recall from \cite[Thm 5.1]{TVV:15} a localization result for the $L^2$-norm on pairs of elements:
\begin{equation}
 \label{loc-L2}
 \inf_{v\in \FEspace}\norm{u-v}
 \lesssim \left(
  \sum_{\fa\in\faces}
  \inf_{V\in\LFEspace{\pair{\fa}}} \norm{u-V}^2_{\pair{\fa}}
 \right)^{\frac12}.
\end{equation}

Combining \eqref{dec-norms} with \eqref{loc-L2} and \eqref{loc-res-grad} we derive the following result.
\begin{prop}
Assume the coefficient $a$ satisfies the quasi-monotonicity property \ref{D:quasi-mon}. Then, for every $u\in H^1(\Omega)$ there holds
\begin{equation*}
\begin{aligned}
&\inf_{v\in S^{\ell,0}(\tri)}\norm{a^{1/2}\nabla(u-v)}^2+\inf_{w\in S^{\ell,0}(\tri)}\beta\norm{u-w}^2
\\&\qquad\approx
\sum_{\fa\in\faces}\inf_{Q\in S|_{\omega_\fa}}\beta\norm{u-Q}^2_{\omega_\fa}
+\sum_{\el\in\tri}\inf_{P\in \Poly_\ell(\el)}\norm{a^{1/2}\nabla(u-P)}^2_{\el},
\end{aligned}
\end{equation*}
where the hidden constant is independent of $\alpha$ and $\beta$.
\end{prop}

\section{Without quasi-monotonicity localization is not robust}
\label{S:not-rob}
In this section we show with two examples that, without the quasi-monotonicity property \ref{D:quasi-mon}, a robust localization is not possible neither on elements, nor on pairs of elements, nor on stars of elements sharing a common vertex. 

\subsection{Localization on elements and pairs of elements}
\label{S:single-pair}
We first consider the localization on single elements. To this end consider the domain 
\begin{equation*}
\Omega:=\{(x,y)\mid -1\leq x,y\leq 1,\ -1\leq x+y\leq 1 \}\subset\R^2
\end{equation*}
depicted in Figure \ref{counter_ex}, together with its triangulation $\el_1,\ldots,\el_6$ and the function $a$ defined by
\begin{equation*}
 a=\left\{
\begin{aligned}
& 1 \qquad && \text{if } x,y>0\ \text{or }x,y<0,\\
& \veps^2 &&  \text{otherwise,}
\end{aligned}
\right.
\end{equation*}
with $0 < \veps \ll 1$. Note that, in this configuration, the quasi-monotonicity property does not hold.
\newrgbcolor{zzttqq}{0.8 0.8 0.8}
 \begin{figure}[htb]
  \begin{center}
\begin{tikzpicture}[scale=0.8](-3.1,-3.1)(3.1,3.1)
 \draw[fill=zzttqq]  (-2,0)--(0,0)--(0,2)--(-2,2);
\draw[fill=zzttqq]   (2,0)--(0,0)--(0,-2)--(2,-2);
 \draw (0,-2)--(-2,0) --(2,0)--(0,2)--(0,-2)--(2,-2)--(2,0);
  \draw (-2,0)--(-2,2)--(0,2);
  \draw (-2,2)--(2,-2);
\draw[->, dashed] (0,-3)--(0,3);
\draw[->,dashed] (-3,0)--(3,0);
\draw[dashed] (0.8,0) arc [radius=0.8, start angle=0, end angle= 90];
\draw[dashed] (0,-0.8) arc [radius=0.8, start angle=270, end angle=180];
\node [above right] at (1,1) {$K_1$};
\node [below left] at (-1,-1){$K_4$};
\node [above] at(-1,2.1){$K_2$};
\node[left] at (-2.1,1){$K_3$};
\node[below]at (1,-2.1){$K_5$};
\node[right]at(2.1,-1){$K_6$};
\node [below] at (0.8,0) {$\veps$};
\end{tikzpicture}
\caption{Domain $\Omega$ and its subdivision in six triangles. Gray areas: $a=\veps^2$, white areas $a=1$.}\label{counter_ex}
 \end{center}
 \end{figure}
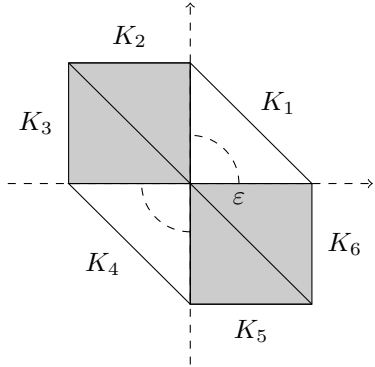

 Define $w(x,y):=1-x-y$ and
\begin{equation*}
 \rho_\veps(r):=\left\{\begin{aligned}
                       &(1-\veps)\left(\frac{r}{\veps}\right)^{\veps} && \text{if }0 \le r < \veps
			\\
			& 1-r && \text{if }\veps \le r \le 1 \\
			& 0 && \text{if } 1 < r
                      \end{aligned} 
\right.,
\end{equation*}
and observe that 
\begin{subequations}
\begin{align}
\label{grad-rho}
\norm{\nabla \rho_\veps}^2_{B(0,\veps)}&=2\pi\int_0^\veps r\left(\frac{\partial \rho_\veps}{\partial r}\right)^2
=\pi\veps(1-\veps)^2,\\[2mm]
\int_{0}^1 \frac1r\rho^2_\veps&\leq\int_{0}^\veps \frac{r^{2\veps-1}}{\veps^{2\veps}}+\int_{\veps}^1\frac{1}r=\frac1{2\veps}-\log \veps.
\label{int-rho-r-1}
\end{align}
\end{subequations}
Moreover let $v_\veps$ satisfy
\begin{equation*}
\begin{aligned}\Delta v_\veps&=0\quad\qquad &&\text{in } K_1\setminus B(0,\veps),\\
 v_\veps&=w &&\text{on }\partial(K_1\setminus B(0,\veps))\setminus \partial B(0,\veps),\\
v_\veps&=1-\veps &&\text{on }\partial B(0,\veps)\cap K_1.
\end{aligned}
\end{equation*}
We take as target function $u_\veps\in H^1_0(\Omega)$ defined as follows 
\begin{equation*}
\begin{aligned} 
u_\veps&=\left\{\begin{aligned}
                 &\rho_\veps(r)\bigg(3-\frac4\pi\theta\bigg) && \text{in } K_2\cup K_3\\[2mm]
		 &\rho_\veps(r) &&\text{in }K_1\cap B(0,\veps)\\
		 & v_\veps &&\text{in } K_1\setminus B(0,\veps)
                \end{aligned}
 \right.\\
u_\veps(x,y)&=-u_\veps(-x,-y), \qquad\forall (x,y)\in K_4\cup K_5\cup K_6.
\end{aligned}
\end{equation*}
Thanks to the symmetry of $u_\veps$, we have, for every $Q\in S^{1,0}(\tri) \cap H^1_0(\Omega)$, 
\begin{equation*}
\veps^2\int_{K_2\cup K_3\cup K_5\cup K_6} \nabla u \cdot\nabla Q+\int_{K_1\cup K_4}\nabla u\cdot\nabla Q=0,
\end{equation*}
so that the global Ritz projection $Ru_\veps$ of $u_\veps$ onto 
$S^{1, 0}(\tri) \cap H^1_0(\Omega)$ is the zero function. In order to prove that a localization result does not robustly hold, 
we show that $\tnorm{u_\veps-Ru_\veps}^2>0$ uniformly in $\veps$, while 
\begin{equation}
\label{SRitz-pairs-->0}
\sum_{j=1}^6\tnorm{u_\veps-R_ju_\veps}^2_{K_j} \lesssim \veps,
\end{equation} 
where $R_ju_\veps$ is the Ritz projection of $u_\veps$ onto the space of affine functions on the element $K_j$.
\\
To this end, we recall that $0 < \veps \ll 1$ and  first 
show that $w$ is a good approximation to $u_\veps$ on $K_1$. 
Observe that $u_\veps-w$ is harmonic in $K_1\setminus B(0,\veps)$ and 
\begin{equation*}
\begin{aligned}
u_\veps-w&=0 &&\text{on }\partial(K_1\setminus B(0,\veps))\setminus \partial B(0,\veps),\\ 
u_\veps-w&=\veps g(\theta) &&\text{on }\partial B(0,\veps)\cap K_1,
\end{aligned}
,
\end{equation*} 
where $g(\theta):=\cos\theta+\sin\theta-1$. Consider the function $\widetilde{u}\in W^{1,\infty}(K_1\setminus B(0,\veps))$ defined by
\begin{equation*}
\widetilde{u}(r,\theta):= g(\theta)\frac{1-r\cos\theta-r\sin\theta}{1-\veps\cos\theta-\veps\sin\theta}
\end{equation*} 
and such that $\veps\widetilde{u}$ equals $u_\veps-w$ on the boundary $\partial (K_1\setminus B(0,\veps))$. Since $u_\veps-w$ is harmonic, we have
\begin{equation*}
\begin{aligned}
&\norm{\nabla(u_\veps-w)}^2_{K_1\setminus B(0,\veps)}\leq
\norm{\nabla \veps\widetilde{u}}^2_{K_1\setminus B(0,\veps)}\\
&\qquad\lesssim \veps^2
\norm{\frac{\partial \widetilde{u}}{\partial r}}^2_{\infty}+\veps^2|\log \veps|\norm{\frac{\partial \widetilde{u}}{\partial \theta}}^2_{\infty}
\to 0\qquad \text{as }\veps\to 0,
\end{aligned}
\end{equation*}
where $\norm{\cdot}_\infty$ indicates the $L^\infty$-norm. 
This, together with \eqref{grad-rho}, implies that 
\begin{equation}
\begin{aligned}
 &\tnorm{u_\veps-Ru_\veps}^2_{K_1}\leq \tnorm{u_\veps-w}^2_{K_1}\\
&\qquad\leq 
\norm{\nabla(u_\veps-w)}^2_{K_1\setminus B(0,\veps)}+2\norm{\nabla u_\veps}^2_{K_1\cap B(0,\veps)}
\\
&\qquad\qquad
+2\norm{\nabla w}^2_{K_1\cap B(0,\veps)} \lesssim \veps.
\end{aligned}
\label{Ritz-pair-->0}
\end{equation}
Moreover, we also see that the zero function is not a good approximation on $K_1$. In fact
\begin{equation*}
 \norm{\nabla u_\veps}^2_{K_1}\geq\norm{\nabla u_\veps}^2_{K_1\setminus B(0,\veps)}\geq \frac12\norm{\nabla w}^2_{K_1\setminus B(0,\veps)}
-\norm{\nabla(u_\veps-w)}^2_{K_1\setminus B(0,\veps)}\gtrsim 1
\end{equation*}
uniformly in $\veps$.
In particular then 
\begin{equation}\label{norm-u_eps}
 \tnorm{u_\veps-Ru_\veps}^2_\Omega=\tnorm{u_\veps}^2_\Omega\geq 
  \norm{\nabla u_\veps}^2_{K_1} \gtrsim 1
\end{equation}
independently of $\veps$. 
On the other hand, recalling \eqref{int-rho-r-1} we derive
\begin{equation}\label{u_eps-to-0}
\begin{aligned}
&\tnorm{u_\veps}^2_{K_2\cup K_3\cup K_5\cup K_6}=2\tnorm{u_\veps}^2_{K_2\cup K_3}=
2\veps^2\norm{\nabla u_\veps}^2_{K_2\cup K_3} 
\\
&\qquad=2\veps^2\int_0^1\int_{\frac\pi2}^\pi r\left(\frac{\partial \rho_\veps}{\partial r}\right)^2\bigg(3-\frac4\pi\theta\bigg)^2
+\frac{1}{r}\left(\frac4\pi\rho_\veps\right)^2 \lesssim \veps.
\end{aligned}
\end{equation}
This, together with \eqref{Ritz-pair-->0} and its counterpart on $K_4$
leads to \eqref{SRitz-pairs-->0}.

Since on pairs of elements the diffusion is always quasi-monotone, the above example also shows that, without the quasi-monotonicity property \ref{D:quasi-mon}, the localization on pairs of elements is not robust as well.

\subsection{Localization on stars of elements}
\label{S:stars}
In this sub-section we modify the previous example to show that without quasi-monotonicity even the localization on stars of elements sharing a common vertex is not robust. To this end we consider the unit square $\Omega = [0,1]^2$ and subdivide it into $4 N^2$ squares with sides of length $\frac{1}{2 N}$. Each square is further subdivided into two isosceles right-angled triangles by drawing its diagonal from the top-left corner to the bottom-right corner. This defines the triangulation $\tri$. The diffusion $a$ has a checkerboard pattern with the value $a = \frac{1}{N^2}$ on the black squares and the value $a = 1$ on the white squares with the outmost top-left square being black, see the left part of Figure \ref{counter_ex_m}. The partition of $\Omega$ into the small squares is sub-ordinate to a macro-partition consisting of $N^2$ macro-squares with sides of length $\frac{1}{N}$. On each of these macro-squares the diffusion is as in the right part of Figure \ref{counter_ex_m}. Obviously the right part of Figure \ref{counter_ex_m} is similar to Figure \ref{counter_ex} with two additional right-angled isosceles triangles adjacent to the long sides of triangles $K_1$ and $K_4$. Moreover, each macro-square is a translated and rescaled version of the square $[-1,1]^2$.

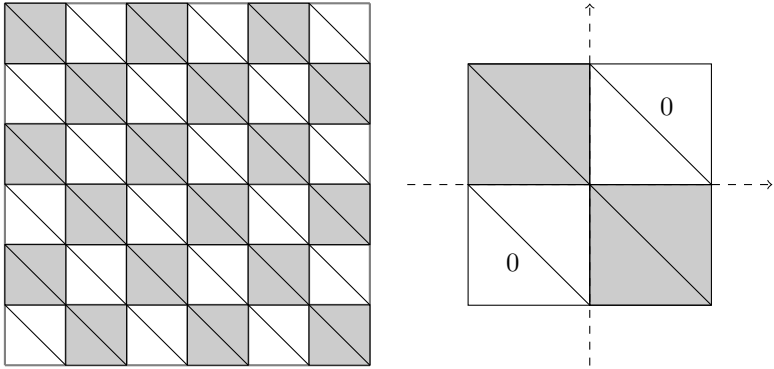
\begin{figure}[htb]
  \begin{center}
\begin{tikzpicture}[scale=0.8](-3.1,-3.1)(3.1,3.1)
\draw [help lines, thick] (-3,-3) grid (3,3);
\foreach \y in {2, 0, -2} {\foreach \x in {-3, -1, 1} \draw [fill=zzttqq] (\x,\y) rectangle (\x+1,\y+1);};
\foreach \y in {1, -1, -3} {\foreach \x in {-2, 0, 2} \draw [fill=zzttqq] (\x,\y) rectangle (\x+1,\y+1);};
\foreach \y in {-2, -1, 0, 1, 2, 3} \draw (-3,\y) -- (\y,-3);
\foreach \x in {-2, -1, 0, 1, 2} \draw (\x,3) -- (3,\x);
\end{tikzpicture}
\quad
\begin{tikzpicture}[scale=0.8](-3.1,-3.1)(3.1,3.1)
 \draw[fill=zzttqq]  (-2,0)--(0,0)--(0,2)--(-2,2);
\draw[fill=zzttqq]   (2,0)--(0,0)--(0,-2)--(2,-2);
\draw (0,0)--(-2,0)--(-2,-2)--(0,-2);
\draw (0,0)--(2,0)--(2,2)--(0,2);
 \draw (0,-2)--(-2,0) --(2,0)--(0,2)--(0,-2)--(2,-2)--(2,0);
  \draw (-2,0)--(-2,2)--(0,2);
  \draw (-2,2)--(2,-2);
\draw[->, dashed] (0,-3)--(0,3);
\draw[->,dashed] (-3,0)--(3,0);
\node[above right] at (1,1) {$0$};
\node[below left] at (-1,-1) {$0$};
\end{tikzpicture}
\caption{Left: triangulation $\tri$ corresponding to $N = 3$. Right: re-scaled macro-square. Gray areas: $a=\frac{1}{N^2}$. White areas: $a=1$. $u_{N,i,j} = 0$ on triangles marked $0$.}\label{counter_ex_m}
 \end{center}
 \end{figure} 

The function $u_\veps$ of the previous sub-section vanishes on the boundary of the domain depicted in Figure \ref{counter_ex}. We may therefore extend it by zero to the square $[-1,1]^2$. This extension is again denoted by $u_\veps$. For $1 \le i,j \le N$ we set
\begin{equation*}
u_{N,i,j}(\cdot) = u_{\frac{1}{N}} \left( 2N \Bigl( \cdot - \bigl( \frac{2i-1}{2N},\frac{2j-1}{2N} \bigr) \Bigr) \right).
\end{equation*}
Since $\tnorm{\cdot}$ is a pure diffusion-norm and the diffusion is piecewise constant on $\tri$, we obtain from \eqref{norm-u_eps} for all $1 \le i,j \le N$ and all $N$
\begin{equation}\label{norm-u_eps-rescaled}
\tnorm{u_{N,i,j}}^2_{(\frac{2i-1}{2N},\frac{2j-1}{2N})+[-\frac{1}{2N},\frac{1}{2N}]^2} = \tnorm{u_{\frac{1}{N}}}^2_{[-1,1]^2} \gtrsim 1.
\end{equation}
We now set
\begin{equation*}
U_N = \sum_{i=1}^N \sum_{j=1}^N \frac{1}{N} u_{N,i,j}.
\end{equation*}

Up to rotation, the patches $\omega_z$ associated with the interior vertices $z \in \verts_\Omega$ of $\tri$ consist of three types. The first type is depicted in the left part of Figure \ref{counter_ex_p}. It is as in Figure \ref{counter_ex} and corresponds to vertices of the form $z = (\frac{2i-1}{2N},\frac{2j-1}{2N})$ in the interior of the macro-squares. The second type is depicted in the middle part of Figure \ref{counter_ex_p} and corresponds to vertices of the form $z = (\frac{2i-1}{2N},\frac{j}{N})$ or $z = (\frac{i}{N},\frac{2j-1}{2N})$, i.e. the mid-points of the boundary edges of the macro-squares. The third type is depicted in the right part of Figure \ref{counter_ex_p} and corresponds to vertices of the form $z = (\frac{i}{N},\frac{j}{N})$, i.e. the corners of the macro-squares. On all these patches, the function $U_N$ is anti-symmetric with respect to reflections at the vertex $z$. Hence, the Ritz projection of $U_N$ onto $S^{1, 0}(\tri) \cap H^1_0(\Omega)$ vanishes. Combined with \eqref{norm-u_eps-rescaled} this implies for all $N$
\begin{equation*}
\inf_{v_\tri \in S^{1, 0}(\tri) \cap H^1_0(\Omega)} \tnorm{U_N - v_\tri}^2 = \tnorm{U_N}^2 \gtrsim 1.
\end{equation*}
We claim that
\begin{equation}\label{local-Ritz-to-0}
\sum_{z \in \verts_\Omega} \inf_{v_z \in S^{1, 0}(\tri)|_{\omega_z} \cap H^1_0(\Omega)} \tnorm{U_N - v_z}^2 \lesssim \frac{1}{N}.
\end{equation}
This establishes the announced non-robustness of the localization on stars when the diffusion does not satisfy the quasi-monotonicity property \ref{D:quasi-mon}.

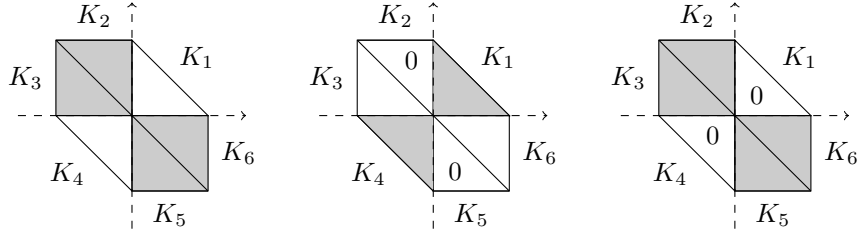
\begin{figure}[htb]
  \begin{center}
\begin{tikzpicture}[scale=0.5](-3.1,-3.1)(3.1,3.1)
 \draw[fill=zzttqq]  (-2,0)--(0,0)--(0,2)--(-2,2);
\draw[fill=zzttqq]   (2,0)--(0,0)--(0,-2)--(2,-2);
 \draw (0,-2)--(-2,0) --(2,0)--(0,2)--(0,-2)--(2,-2)--(2,0);
  \draw (-2,0)--(-2,2)--(0,2);
  \draw (-2,2)--(2,-2);
\draw[->, dashed] (0,-3)--(0,3);
\draw[->,dashed] (-3,0)--(3,0);
\node [above right] at (1,1) {$K_1$};
\node [below left] at (-1,-1){$K_4$};
\node [above] at(-1,2.1){$K_2$};
\node[left] at (-2.1,1){$K_3$};
\node[below]at (1,-2.1){$K_5$};
\node[right]at(2.1,-1){$K_6$};
\end{tikzpicture}
\quad
\begin{tikzpicture}[scale=0.5](-3.1,-3.1)(3.1,3.1)
 \draw  (-2,0)--(0,0)--(0,2)--(-2,2);
\draw   (2,0)--(0,0)--(0,-2)--(2,-2);
 \draw (0,-2)--(-2,0) --(2,0)--(0,2)--(0,-2)--(2,-2)--(2,0);
  \draw (-2,0)--(-2,2)--(0,2);
  \draw (-2,2)--(2,-2);
  \draw[fill=zzttqq] (0,0)--(2,0)--(0,2);
  \draw[fill=zzttqq] (0,0)--(-2,0)--(0,-2);
\draw[->, dashed] (0,-3)--(0,3);
\draw[->,dashed] (-3,0)--(3,0);
\node[above right] at (-1,1) {$0$};
\node[below left] at (1,-1) {$0$};
\node [above right] at (1,1) {$K_1$};
\node [below left] at (-1,-1){$K_4$};
\node [above] at(-1,2.1){$K_2$};
\node[left] at (-2.1,1){$K_3$};
\node[below]at (1,-2.1){$K_5$};
\node[right]at(2.1,-1){$K_6$};
\end{tikzpicture}
\quad
\begin{tikzpicture}[scale=0.5](-3.1,-3.1)(3.1,3.1)
 \draw[fill=zzttqq]  (-2,0)--(0,0)--(0,2)--(-2,2);
\draw[fill=zzttqq]   (2,0)--(0,0)--(0,-2)--(2,-2);
 \draw (0,-2)--(-2,0) --(2,0)--(0,2)--(0,-2)--(2,-2)--(2,0);
  \draw (-2,0)--(-2,2)--(0,2);
  \draw (-2,2)--(2,-2);
\draw[->, dashed] (0,-3)--(0,3);
\draw[->,dashed] (-3,0)--(3,0);
\node[below left] at (1,1) {$0$};
\node[above right] at (-1,-1) {$0$};
\node [above right] at (1,1) {$K_1$};
\node [below left] at (-1,-1){$K_4$};
\node [above] at(-1,2.1){$K_2$};
\node[left] at (-2.1,1){$K_3$};
\node[below]at (1,-2.1){$K_5$};
\node[right]at(2.1,-1){$K_6$};
\end{tikzpicture}
\caption{Types of patches $\omega_z$ associated with interior vertices $z \in \verts_\Omega$ of $\tri$ (up to rotation). Left: vertex $z = (\frac{2i-1}{2N},\frac{2j-1}{2N})$ in the interior of a macro-square. Middle: vertex $z = (\frac{2i-1}{2N},\frac{j}{N})$ or $z = (\frac{i}{N},\frac{2j-1}{2N})$ in the middle of a boundary edge of a macro-square. Right: vertex $z = (\frac{i}{N},\frac{j}{N})$ at a corner of a macro-square. Gray areas: $a=\frac{1}{N^2}$, white areas $a=1$. $U_N = 0$ on triangles marked $0$.}\label{counter_ex_p}
 \end{center}
 \end{figure}

To prove \eqref{local-Ritz-to-0} we first consider a vertex $z = (\frac{2i-1}{2N},\frac{2j-1}{2N})$ in the interior of a macro-square as depicted in the left part of Figure \ref{counter_ex_p}. From \eqref{Ritz-pair-->0} and \eqref{u_eps-to-0} we know that $U_N$ is well approximated by the function $\frac{1}{N} \phi_z$ on triangle $K_1$, by the function $-\frac{1}{N} \phi_z$ on the triangle $K_4$ and the function $0$ on the remaining triangles where $\phi_z$ is the first order nodal shape function associated with the vertex $z$. Unfortunately, this function is discontinuous and thus not contained in $S^{1,0}(\tri)|_{\omega_z}$. Yet, since $\tnorm{\cdot}$ is a pure diffusion norm with piecewise constant diffusion, we may add element-wise a constant without changing the norm. We therefore consider the continuous piecewise affine function $v_z$ that takes the value $-\frac{1}{N}$ at the vertices of $K_1$ different from $z$, the value $\frac{1}{N}$ at the vertices of $K_4$ different from $z$, and the value $0$ at all other vertices in $\omega_z$. From \eqref{Ritz-pair-->0} we conclude
\begin{equation*}
\tnorm{U_N - v_z}^2_{K_1 \cup K_4} = \tnorm{U_N - \frac{1}{N} \phi_z}^2_{K_1} + \tnorm{U_N + \frac{1}{N} \phi_z}^2_{K_4} \lesssim \frac{1}{N^3},
\end{equation*}
whereas \eqref{u_eps-to-0} implies
\begin{equation*}
\tnorm{U_N - v_z}^2_{K_2 \cup K_3 \cup K_5 \cup K_6} \le 2 \tnorm{U_N}^2_{K_2 \cup K_3 \cup K_5 \cup K_6} + 2 \tnorm{v_z}^2_{K_2 \cup K_3 \cup K_5 \cup K_6} \lesssim \frac{1}{N^3}.
\end{equation*}
Next, we consider a vertex $z = (\frac{2i-1}{2N},\frac{j}{N})$ or $z = (\frac{i}{N},\frac{2j-1}{2N})$ in the middle of a boundary edge of a macro-square as depicted as in the middle of Figure \ref{counter_ex_p}. Now, we consider the function $v_z$ which takes the value $\frac{1}{N}$ at the common vertex of triangles $K_3$ and $K_4$ different from $z$, the value $-\frac{1}{N}$ at the common vertex of triangles $K_1$ and $K_6$ different from $z$, and the value $0$ at all other vertices in $\omega_z$. Taking into account that $U_N$ and $v_z$ vanish on the triangles $K_2$ and $K_5$, we obtain with the same arguments as in the previous case
\begin{equation*}
\tnorm{U_N - v_z}^2_{\omega_z} = \tnorm{U_N - v_z}^2_{K_1 \cup K_3 \cup K_4 \cup K_6} \lesssim \frac{1}{N^3}.
\end{equation*}
Finally, we consider a vertex $z = (\frac{i}{N},\frac{j}{N})$ at a corner of a macro-square as depicted in the right part of Figure \ref{counter_ex_p}. We now set $v_z = 0$. Since $U_N$ vanishes on the triangles $K_1$ and $K_4$, we obtain from \eqref{u_eps-to-0}
\begin{equation*}
\tnorm{U_N - v_z}^2_{\omega_z} = \tnorm{U_N}^2_{K_2 \cup K_3 \cup K_5 \cup K_6} \lesssim \frac{1}{N^3}.
\end{equation*}
Since $\verts_\Omega$ consists of $\approx N^2$ vertices, these estimates establish \eqref{local-Ritz-to-0}.

\section*{Acknowledgements}
The first author was founded by the DFG under grant VE 397/1-1 AOBJ:612415 during her stay with the second author's group.

\bibliographystyle{siam}

\begin{thebibliography}{10}

\bibitem{Bernardi.Verfuerth:00}
{\sc C.~Bernardi and R.~Verf{\"u}rth}, {\em Adaptive finite element methods for
  elliptic equations with non-smooth coefficients}, Numer. Math., 85 (2000),
  pp.~579--608.

\bibitem{Binev.DeVore:04}
{\sc P.~Binev and R.~DeVore}, {\em Fast computation in adaptive tree
  approximation}, Numer. Math., 97 (2004), pp.~193--217.

\bibitem{Ciarlet:78}
{\sc P.~G. Ciarlet}, {\em The finite element method for elliptic problems},
  vol.~4 of Studies in Mathematics and its Applications, North--Holland,
  Amsterdam, 1978.

\bibitem{Dryja.Sarkis.Widlund:96}
{\sc M.~Dryja, M.~V. Sarkis, and O.~B. Widlund}, {\em Multilevel schwarz
  methods for elliptic problems with discontinuous coefficients in three
  dimensions}, Numer. Math., 72 (1996), pp.~313--348.

\bibitem{Pechstein.Scheichl:13}
{\sc C.~Pechstein and R.~Scheichl}, {\em Weighted {P}oincar\'e inequalities},
  IMA J. Numer. Anal., 33 (2013), pp.~652--686.

\bibitem{TVV:-1}
{\sc F.~Tantardini, A.~Veeser, and R.~Verf{\"u}rth}, {\em
  {$H^{-1}$}-approximation with piecewise polynomials}.
\newblock Tech.~Rep.~Ruhr-Universit{\"a}t Bochum, October 2016.

\bibitem{TVV:15}
\leavevmode\vrule height 2pt depth -1.6pt width 23pt, {\em Robust localization
  of the best error with finite elements in the reaction-diffusion norm},
  Constr. Approx., 42 (2015), pp.~313--347.

\bibitem{Veeser:16}
{\sc A.~Veeser}, {\em Approximating gradients with continuous piecewise
  polynomial functions}, Found. Comput. Math., 16 (2016), pp.~723--750.

\bibitem{Veeser.Verfuerth:12}
{\sc A.~Veeser and R.~Verf\"urth}, {\em Poincar\'e constants for finite element
  stars}, IMA Journal of Numerical Analysis, 32 (2012), pp.~30--47.

\bibitem{Zanotti:13}
{\sc P.~Zanotti}, {\em Locking and coupling in piecewise polynomial
  approximation}, Master's thesis, Universit{\`a} degli Studi di Milano, 2013.

\end{thebibliography}

\def\cprime{$'$} \def\cprime{$'$}

\end{document}